\documentclass[12pt]{amsart}

\usepackage{amsmath,amsthm,amsfonts,latexsym,amssymb,enumerate,color}
\usepackage{graphicx}
\usepackage{enumitem}

\newcommand\CC{{\mathbb C}}

\newcommand\GF{{\mathcal F}}

\newcommand\DD{{\mathbb D}}
\newcommand\GK{{\mathcal K}}
\newcommand\GJ{{\mathcal J}}
\newcommand\GM{{\mathcal M}}
\newcommand\GN{{\mathcal N}}

\newcommand\GP{{\mathcal P}}
\newcommand\RR{{\mathbb R}}

\newcommand\TT{{\mathbb T}}
\newcommand\LH{{\mathcal L}({\mathcal H})}
\newcommand\LJ{{\mathcal L}({\mathcal J})}
\newcommand\LK{{\mathcal L}({\mathcal K})}

\newcommand\GH{{\mathcal H}}

\newcommand\Id{{\rm Id}}

\def\beq{\begin{equation}}
\def\eeq{\end{equation}}
\newtheorem{thm}{Theorem}[section]
\newtheorem{prop}[thm]{Proposition}
\newtheorem{defn}[thm]{Definition}
\newtheorem{lem}[thm]{Lemma}
\newtheorem{cor}[thm]{Corollary}
\newtheorem{rem}[thm]{Remark}

\newtheorem{ex}[thm]{Example}
\newcommand\beginpf{\noindent {\bf Proof:} \quad}
\newcommand\fini{\rightline{$\Box$}}
\newcommand\re{\mathop{\rm Re}\nolimits}
\newcommand\im{\mathop{\rm Im}\nolimits}
\def\beginpf{\begin{proof}}
\def\endpf{\end{proof}}

\newcommand{\tpos}{{t \geqslant 0}}
\def\C{{\mathbb C}}
\def\N{{\mathbb N}}
\def\R{{\mathbb R}}

\newcommand{\B}{\mathcal{B}}
\def\restriction#1#2{\mathchoice
	{\setbox1\hbox{${\displaystyle #1}_{\scriptstyle #2}$}
		\restrictionaux{#1}{#2}}
	{\setbox1\hbox{${\textstyle #1}_{\scriptstyle #2}$}
		\restrictionaux{#1}{#2}}
	{\setbox1\hbox{${\scriptstyle #1}_{\scriptscriptstyle #2}$}
		\restrictionaux{#1}{#2}}
	{\setbox1\hbox{${\scriptscriptstyle #1}_{\scriptscriptstyle #2}$}
		\restrictionaux{#1}{#2}}}
\def\restrictionaux#1#2{{#1\,\smash{\vrule height .8\ht1 depth .85\dp1}}_{\,#2}}

\newcommand{\Hil}{\mathfrak{H}}

\renewcommand\phi{\varphi}

\DeclareMathOperator{\dom}{dom}

\def\Re{\re}

\begin{document}
\title[Universality and models for semigroups of operators]{Universality and models for semigroups of operators on a Hilbert space}

\author{B. C\'elari\`es} 
\address{Benjamin C\'elari\`es, Universit\'e de Lyon, Universit\'e Lyon 1, Institut Camille Jordan, CNRS UMR 5208, 43 bld du 11/11/1918, F-69622 Villeurbanne} 
\email{celaries@math.univ-lyon1.fr}

\author{I. Chalendar}
\address{Isabelle Chalendar, Universit\'e Paris Est Marne-la-Vall\'ee, 5 bld Descartes, Champs-sur-Marne, 77454 Marne-la-Vall\'ee, Cedex 2  (France)}
\email{isabelle.chalendar@u-pem.fr}

\author{J.R.Partington}
\address{Jonathan R. Partington, School of Mathematics, University of Leeds, Leeds LS2 9JT, UK}
 \email{J.R. Partington@leeds.ac.uk}
 
 %\noindent\textsc{Mathematics Subject Classification} (2000):
 %Primary:  47A15, 47D03
 %Secondary:  30H10, 31C25
 
\baselineskip18pt

\maketitle

\bibliographystyle{plain}

\begin{abstract}
This paper considers universal Hilbert space operators in the sense of Rota, and gives criteria for 
universality of semigroups in the context of uniformly continuous semigroups and contraction semigroups.
Specific examples are given.
Universal semigroups provide models for these classes of semigroups: following a line of research initiated by Shimorin,
models for concave semigroups are developed, in terms of shifts on reproducing kernel Hilbert spaces.
\end{abstract}

\section{Introduction}

In this paper $\GH$ will always  denote a separable infinite-dimensional Hilbert space and $\LH$ the
space of bounded linear operators on $\GH$.

\begin{defn}
An operator $U \in \LH$ is \emph{universal} if, for every $T \in \LH$, there exists  a closed subspace $\mathcal{M}$ of $\GH$ invariant for $U$, a constant $\lambda \in \CC$ and a bounded linear isomorphism $R : \mathcal{M} \rightarrow \GH$ such that 
\[T = \lambda R U_{|\mathcal{M}}R^{-1}.\]
\end{defn}
The concept of a universal operator was introduced by Rota \cite{rota59,rota60} where he showed that the backward shift of infinite multiplicity is an explicit example of such operator. 
The invariant subspace problem provides a  motivation for studying universal operators since every operator has a nontrivial invariant closed subspace if and only if all minimal (with respect to the inclusion) invariant subspaces of any universal operator are of dimension $1$.  
See also \cite[Chap. 8]{CPbook} and \cite{concrete} for further information on this topic. 
More recently, Schroderus and Tylli
 \cite[Thm. 2.2, Cor.2.3]{riikka} have studied universality from the point of view of spectral
 properties of the operator.\\
 
 We first study the Caradus theorem which gives sufficient conditions implying  the universality of an operator. We then introduce the notion of positive universality which is natural in view of producing a consistent definition of universality for a strongly continuous semigroup. \\
 
After an  analysis of a relevant definition for the universality  of a  semigroup, we give a complete answer for uniformly continuous  groups in terms of the universality of the generator. \\   

We then study examples of universal $C_0$-semigroups of contractions and quasicontractions,
and produce a large class of universal semigroups arising from Toeplitz operators with anti-analytic symbol. \\

The  very last section of the paper deals with $C_0$-semigroups which are not quasicontractive. 
Under the conditions of concavity and analyticity, which imply the existence of a Wold-type decomposition, we can provide models for such semigroups.

\section{Universality of an operator}

Surprisingly, there are many universal operators since Caradus gave a large class of operators (defined below) with this property.  
\begin{defn}\label{caradus}
Let $U \in \LH$. We say that $U$ is a {\em Caradus operator} if it satisfies the conditions:
\begin{enumerate}
\item[(i)] $\ker U$ is infinite-dimensional;
\item[(ii)] $U$ is surjective.
\end{enumerate}
\end{defn}
Caradus \cite{caradus} proved that every Caradus operator is universal.

The standard example of a Caradus operator (given by Rota) is the backward
shift of infinite multiplicity, which can also be realised as the backward
shift $S_{1}$ on $L^2(0,\infty)$, defined almost everywhere by
\[
S_{1}f(t)=f(t+1), \qquad (t \ge 0)
\]
for $f \in L^2(0,\infty)$.

The condition that $ \ker U$ is infinite-dimensional is clearly necessary for universality, but
surjectivity is not (as can be seen by taking a direct sum of a universal operator with any
other operator).
However, if $U$ is universal, then   $U_{|\GM}$ is similar to a multiple of the backward
shift  for some invariant
subspace $\GM$, and thus $U$ has a restriction that is a Caradus operator.

The proof of Caradus's theorem in fact shows that Caradus operators have the
formally stronger property of {\rm positive universality}, defined as follows. 
\begin{defn}
An operator $U \in \LH$ is \emph{positively universal} if, for every $T \in \LH$, there exists  a closed subspace $\mathcal{M}$ of $\GH$ invariant for $U$, a constant $\lambda \ge 0$ and a bounded linear isomorphism $R : \mathcal{M} \rightarrow \GH$ such that 
\[T = \lambda R U_{|\mathcal{M}}R^{-1}.\]
\end{defn}

In fact positive universality is equivalent to universality, as the following result shows.

\begin{prop}\label{prop:equiv}
Let $U \in \LH$. Then $U$ is universal if and only if it is positively universal.
\end{prop}

\beginpf
Let $V \in \LJ$ be an arbitrary positively universal operator as given by Cara\-dus's theorem,
e.g. the backward shift on $L^2(0,\infty)$, so that $\alpha V$ is also 
positively universal if $\alpha \in \CC \setminus \{0\}$.

Now there is an invariant subspace $\GM$ for $U$, and $\alpha \in \CC \setminus \{0\}$, so that we can write
$U_{|\GM} = \alpha R^{-1} VR$ with $R:\GM \to \GJ$ an isomorphism.
Then $U_{|\GM}= |\alpha| R^{-1} WR$, where $W=\alpha V/|\alpha|$, which is positively universal.

Finally, if $T \in \LK$ is any operator, then we can write $W_{|\GN}=\lambda Q^{-1}  TQ$, where $\GN$ is invariant for $W$, $Q: \GN \to \GK$ is an isomorphism,
and $\lambda>0$.

So $U_{|\GP} = \lambda |\alpha| R^{-1}Q^{-1} T Q R$, where $\GP=R^{-1}(\GN)$ is invariant for $U$, and so
$U$ is positively universal.
\endpf

It was shown by Rota \cite{rota60} that the backward shift $S_1$ of infinite multiplicity
also has the property of 1-universality for all operators $T \in \LH$ of spectral radius
strictly less than 1; that is, such an operator can be written as 
\[
T= R S_{1}{}_{|\GM} R^{-1},
\]
where $\GM$ is an invariant subspace for $S_1$ and $R: \GM \to \GH$ is an isomorphism.
Another famous example of a universal operator is due to E. Nordgren, P. Rosenthal and F. Wintrobe \cite{NRW} who 
proved that $C_\varphi-\Id$ is universal on the Hardy space $H^2(\DD)$, with $\varphi$ is a hyperbolic automorphism of the unit disc. E. Pozzi \cite{PP11,elodie} studied universal shifts and weighted composition operators on various spaces, and C. Cowen and   E. Gallardo-Guti\'errez produced examples of universal anti-analytic Toeplitz operators \cite{concrete}.

 \section{Universal semigroups}
 \subsection{Basic facts on semigroups}
 A family $(T_t)_{t\geq 0}$ in  $\LH$ is called a \emph{$C_0$-semigroup} if 
 \begin{center}
$T_0=\Id$, $T_{t+s}=T_tT_s$ for all $s,t\geq 0$ and $\forall x\in\GH$, $\lim_{t\to 0}T_t x=x$. 
\end{center}
A \emph{uniformly continuous semigroup} is a $C_0$-semigroup such that 
 \[\lim_{t\to 0}\|T_t-\Id\|=0.\]
 Recall also that the generator of a $C_0$-semigroup denoted by $A$ is defined by 
 \[Ax=\lim_{t\to 0}\frac{T_tx-x}{t}\]
 on $D(A):=\{x:\lim_{t\to 0}\frac{T_tx-x}{t}\mbox{ exists}\}$. Moreover $(T_t)_{t\geq 0}$ is uniformly continuous if and only if $D(A)=\GH$ , that is, if and only if $A\in\LH$.  See for example \cite{EN} for an introduction to $C_0$-semigroups.
 
 Since a $C_0$-semigroup $(T_t)_{t\geq 0}$ is not always uniformly continuous, its generator $A$ is 
in general  an unbounded  operator. 
Nevertheless, provided that $1$ is not in the spectrum of $A$, the (negative) Cayley transform of $A$ defined by $V:=(A+\Id)(A-\Id)^{-1}$ is a bounded operator and is called the \emph{cogenerator} of $(T_t)_{t\geq 0}$.    
 In \cite[Thm III.8.1]{nagy-foias} the following equivalence is proved:
 \begin{center}
 	$V\in \LH$ is the cogenerator of a $C_0$-semigroup of contractions if and only if $V$ is a contraction and $1$ is not an eigenvalue of $V$. 
 \end{center}
 Not only contractivity is preserved by the cogenerator. Indeed, Sz.-Nagy and Foias \cite[Prop. 8.2]{nagy-foias} proved that a $C_0$-semigroup of contractions consists of normal,  self-adjoint, or unitary operators, if and only if its cogenerator is normal, self-adjoint, or unitary, respectively.     
 
 \subsection{Definitions of universality for semigroups}
 Let $(S_t)_\tpos$ be the $C_0$-semigroup on $L^2([0,+\infty))$ such that for all $\tpos$, \[S_t: \left\lbrace \begin{array}{rll}
 L^2([0;+ \infty)) & \rightarrow & L^2([0+ \infty)) \\ 
 f & \mapsto & f(\cdot + t)
 \end{array}  \right. .\]  
 
 Then for any $t>0$,  by Caradus' theorem, $S_t$ is universal.\\ 
 
 Therefore, for any $C_0$-semigroup $(T_t)_\tpos$ on $L^2([0,+\infty))$, there exist sequences  $({\mathcal M}_t)_t$ of closed subspaces of $L^2([0,+\infty))$, $(\lambda_t)_t$ of complex numbers and $(R_t)_t$ of bounded isomorphisms from ${\mathcal M}_t$ onto $L^2([0,\infty))$ such that, for every $t>0$, 
 \[  T_t=\lambda_t R_t(S_t)_{|{\mathcal M}_t}R_t^{-1}.  \]
 
 This possible definition of universal semigroups is not fully satisfactory since $\lambda_t$, $\mathcal{M}_t$, and $R_t$ depend  heavily on $t$. 
 
 A much more natural and appropriate definition is the following. 
 
 \begin{defn}
 	Let $(U_t)_\tpos$ be a $C_0$-semigroup (resp. uniformly continuous) on a Hilbert space $\GH$. It is called a \emph{universal} $C_0$-semigroup (resp. uniformly continuous) if for every $C_0$-semigroup $(T_t)_\tpos$ (resp. uniformly continuous), there exist a closed subspace $\mathcal M$ invariant by every $(U_t)_\tpos$, $\lambda \in \R$, $\mu \in \R^{+*}$, and $R:{\mathcal M}\to  \GH$ a bounded isomorphism such that, for all $\tpos$: 
 	\[ T_t  =R(e^{\lambda t}U_{\mu t})_{|{\mathcal M}}R^{-1}.   \]  	
 \end{defn}
Using this definition of universality for semigroups,  
 a certain amount of caution is required: 
for the backward shift semigroup on $L^2(0,\infty)$ each $S_t$ is universal, but the semigroup as a whole is
not,
as we shall see later. 
 
 \subsection{Uniformly continuous groups}
 It is very natural to find  a criterion involving the generator which captures all the information pertaining to the semigroups. The easiest case to deal with is when the semigroup is uniformly continuous since its generator is bounded. In this
 situation the semigroup extends to a 
group parametrised by $\RR$.  
 \begin{thm}\label{thm:lambdazero}
	Let $(U_t)_{t \in \RR}$ be a uniformly continuous  group whose (bounded) generator is denoted by  $A$. The following assertions are equivalent:
	\begin{enumerate}
		\item[(i)] for every  uniformly continuous group $(T_t)_{t \in \RR}$, there exists a closed subspace $\mathcal M$ invariant for $(U_t)_{t \in \RR} $, $\mu \ge 0$, and $R:{\mathcal M}\to  \GH$ a bounded isomorphism such that, for all $t \in \RR$: 
\[ T_t  =R U_{\mu t}{}_{|\GM}R^{-1}.   \] 
		\item[(ii)] $A$ is universal.
	\end{enumerate}
\end{thm}

\beginpf
$(i)\Rightarrow (ii)$: Let $B$ be a bounded operator on $\GH$ and  $(T_t)_{t \in \RR}$ be the uniformly continuous semigroup generated by 
$B$. Let ${\mathcal M}$ be a closed subspace of $\GH$, $\mu\ge  0$ and $R:{\mathcal M}\to \GH$ an isomorphism such that 
\[  T_t=R(U_{\mu t})_{|{\mathcal M}}R^{-1}.  \]  
For all $x\in \GH$, we can differentiate $\varphi:t\mapsto T_tx$ at $t=0$ and we get:
\[  Bx=R(\mu  A_{|{\mathcal M}} )R^{-1}x,  \]
which proves that $B$ is   universal. \\
$(ii)\Rightarrow (i)$:  Let $(T_t)_{t \in \RR}$ be a uniformly continuous semigroup whose generator is denoted by $B$. Since $A$ is positively universal by Prop. \ref{prop:equiv}, there exist     a closed subspace ${\mathcal M}$ of $\GH$, $\mu \ge 0$ and $R:{\mathcal M}\to \GH$ an isomorphism such that 
\begin{equation}\label{eq:BA}
B=\mu R(A_{|{\mathcal M}})R^{-1}. 
 \end{equation}  
It follows that, for all ${t \in \RR}$, 
\[e^{tB} = R(e^{\mu tA})_{|{\mathcal M}}R^{-1},\] 
and then $T_t=R(U_{\mu t})_{|{\mathcal M}} R^{-1}$.  
\endpf
 
 \begin{ex}
 Take $A=S_1$. To calculate the semigroup $(U_t)_{t \in \RR}$ it is convenient to
 work with the Fourier transform $\GF$, which, by the Paley--Wiener theorem \cite{rudin2}
 provides an isometric isomorphism between
 $L^2(0,\infty)$ and the Hardy space $H^2(\CC^+)$ of the upper half-plane $\CC^+$.
 Then $S_1^*$ is the right shift by 1 on $L^2(0,\infty)$, and the 
 operator $\GF S_1^* \GF^{-1}$ is the analytic Toeplitz operator with symbol $z \mapsto e^{iz}$.
 
 That is, for $t \in \RR$, $\GF U^*_t  \GF^{-1}$ is the analytic Toeplitz operator with symbol $x\mapsto \exp(t e^{ix})$, where $x \in \RR$,
 and $\GF U_t  \GF^{-1}$ is the anti-analytic Toeplitz operator with symbol
 $x \mapsto \exp(t e^{-ix})$.
 \end{ex}
 
 Note that the shift semigroup $(S_t)_{t \ge 0}$ on $L^2(0,\infty)$ is not universal even for the class
 of all uniformly continuous contraction semigroups.
 Its infinitesimal generator $A$ is  defined by $Af=f'$ and hence $\ker (A-\lambda I)$
 has dimension at most 1 for every $\lambda \in \CC$. Hence if $B$ is a non-zero bounded
 operator with kernel of dimension  at least 2, then we cannot have an identity of the form
$ B-\lambda I=\mu R(A_{|{\mathcal M}})R^{-1}$, 
 and so we do not have an identity of the
 form 
 $e^{tB} = e^{\lambda t}R(S_{\mu t})_{|{\mathcal M}}R^{-1}$.
 
 \subsection{Contraction semigroups}

 %Cogenerators, see \cite{EZ}.
 
 Note that a subspace $\GM$ is invariant for the cogenerator if and only
 if it is invariant for every member of the semigroup \cite{fuhrmann}.
 
 The following theorem \cite[Thm. 8.1.5]{CPbook} can be traced back to \cite{nagy-foias}.
Recall that an operator $T \in \LH$ is said to be $C_{0.}$ if $\|T^n x\| \to 0$ for all $x \in \GH$.
 
 \begin{thm}
 Let $T \in \LH$ be a contraction operator of class $C_{0.}$. Then there
 is an invariant subspace $\GM$ of $S_1$ such that $T$ is unitarily
 equivalent to $S_1{}_{|\GM}$.
 \end{thm}
 
 This easily implies the following result.

 \begin{thm}\label{thm:S1univ}
 Let $(U_t)_{t \ge 0}$ be the semigroup on 
 $\GH=L^2(0,\infty)$ whose cogenerator is $S_1$.
 Then for every $C_{0.}$ contraction semigroup $(T(t))_{t \ge 0}$ 
 on a Hilbert space $\GH$
 there is a common invariant subspace $\GM$ for $(U_t)_{t \ge 0}$ 
 and an isomorphism $R: \GM \to \GH$ such that
 $T(t)=R U_t{}_{|\GM} R^{-1}$ for all $t \ge 0$.
 \end{thm} 
 
 \beginpf
 Consider the cogenerator $W$ of $(T(t))_{t \ge 0}$. This is a $C_{0.}$ contraction,
 by \cite[Sec. III.8--9]{nagy-foias}, and thus can be written as
 $W=R S_1{}_{|\GM} R^{-1}$ for some invariant subspace $\GM$ of $S_1$ and 
 isomorphism $R: \GM \to \GH$. 
 The result then follows by standard calculations.
 \endpf
 
 This semigroup can also be expressed using co-analytic Toeplitz
 operators on the Hardy space $H^2(\CC^+)$. For, with $\GF$
 denoting the Fourier transform once more, we have
 $\GF S_1^* \GF^{-1}$ is the multiplication operator (analytic Toeplitz operator)
 with symbol $e^{iz}$, and thus $\GF U^*_t \GF^{-1}$ has symbol
 \[
 \exp(t (e^{iz}+1)/(e^{iz} -1))= \exp(-it \cot (z/2)).
 \]
 
 If a semigroup $(U_t)_{t \ge 0}$ is quasicontractive, i.e.,
 it satisfies $\|U(t)\| \le e^{\omega t}$ for some $\omega \in \RR$, then
 clearly $( e^{-\lambda t}U(t))_{t \ge 0}$ is a $C_{0.}$ contractive
 semigroup provided that $\lambda > \omega$. We therefore
 have the following corollary.
 
 \begin{cor} 
 Let $(U_t)_{t \ge 0}$ be the semigroup on 
 $\GH=L^2(0,\infty)$ whose cogenerator is $S_1$.
 Then for every   quasicontractive semigroup $(T(t))_{t \ge 0}$ 
 on a Hilbert space $\GH$
 there is a common invariant subspace $\GM$ for $(U_t)_{t \ge 0}$, a constant $\lambda \in \RR$,
 and an isomorphism $R: \GM \to \GH$ such that
 $T(t)=e^{\lambda t} R U_t{}_{|\GM} R^{-1}$ for all $t \ge 0$.
 \end{cor} 
 
 Note that the backward shift semigroup $(\widetilde S_t)_{t \ge 0}$ on $L^2(0,\infty; \GH)$
 is also universal in this sense: 
 see \cite[Thm. 10-18]{fuhrmann}. Note that the example in Theorem \ref{thm:S1univ}
 is defined on the simpler space $L^2(0,\infty)$.

%{\bf Discussion of unitary semigroups. No universality in this case unless absolutely continuous???
%How about groups???}

The operator $S_1$ is the adjoint of a completely non-unitary unilateral right shift of infinite multiplicity.
There are many Toeplitz operators that are unitarily equivalent to it, and thus have
similar properties.

The following result is well-known, and we give a simple proof to illustrate it. We shall perform
calculations on the Hardy space $H^2(\DD)$ of the disc, but analogous results hold for Hardy spaces of the half-plane.

\begin{lem}
Let $\phi$ be an inner function. Then the analytic Toeplitz operator $T_\phi$ is unitarily 
equivalent to a unilateral right shift of multiplicity $\dim K_\phi$, where $K_\phi= H^2 \ominus \phi H^2$.
\end{lem}
\beginpf
This follows easily from the orthogonal decomposition
\[
H^2 = K_\phi \oplus \phi K_\phi \oplus \phi^2 K_\phi \oplus \cdots,
\]
which has been used in many places, for example, \cite{cgp15}.
\endpf

If we take $\phi$ to be irrational (not a finite Blaschke product), then
$V=T_\phi^*$ is the cogenerator of a $C_0$ semigroup on $H^2$, and it is easy to check
that $\exp(t(\phi+1)/(\phi-1))$ is a singular inner function for each $t \ge 0$. We therefore
have the following theorem.

\begin{thm}
(i) Let $\phi$ be an inner function that is not a finite Blaschke product. Then the semigroup $(U_t)_{t \ge 0}$
consisting of anti-analytic Toeplitz operators $T^*_{\phi_t}$,
where 
\[
\phi_t = \exp \left( t \frac{\phi+1}{\phi-1} \right), \qquad t \ge 0,
\]
 is universal for the class of $C_{0.}$ contraction semigroups.
 
(ii) Moreover, if a semigroup $(U_t)_{t \ge 0}$ has the form
$U_t= T^*_{\phi_t}$, where $\phi_t=\exp(t \psi)$ is a singular inner function
for each $t$, then $\phi:= (\psi+1)/(\psi-1)$ is inner, and if it is irrational the conclusions of part (i) apply.
\end{thm}

Note that the semigroup corresponding to $\phi(z)=-z$ (inner, but rational) is given by the function $\phi_t=\exp(t(1-z)/(1+z))$. This
 is unitarily equivalent to the shift semigroup $(S_t)_{t \ge 0}$, which is not universal.\\

 \begin{rem}{\rm 
 It was shown by Gamal' \cite{gamal1,gamal2}, extending work of Clark \cite{clark}, that if
 $B$ is a finite Blaschke product and $\phi$ is an irrational inner function, then the
 Toeplitz operator $T_{\phi/B}$ is similar to an isometry $U \oplus S$, where $U$ is unitary
 and $S$ is a unilateral shift of infinite multiplicity.
 It follows that the semigroup with cogenerator $T_{B/\phi}$ is universal for
 the class of contraction semigroups, in the sense of Theorem \ref{thm:S1univ}.}
 \end{rem}
 
 \begin{rem}
 It was shown by Sz.-Nagy \cite{nagy} that every bounded $C_0$ group on a Hilbert space is similar to
 a group of unitary operators. One might therefore hope for the existence of a universal unitary
 group $(U_t)_{t \in \RR}$ such that every bounded group $(T_t)_{t \in \RR}$ could be
 represented in the form $T_t = R (U_t)_{|\GM} R^{-1}$ for some isomorphism $R$ and 
 invariant subspace $\GM$ for $(U_t)$. However, by looking at cogenerators, we see that $(U_t)$ would
 possess a unitary cogenerator such that every point on $\TT$ with the exception of $1$ would
 be an eigenvalue of infinite multiplicity. In a separable Hilbert space this is impossible, since
 eigenvectors corresponding to distinct eigenvalues are orthogonal.
 \end{rem}

\section{$C_0$-semigroups close to isometries}
Let $\GH$ be a complex infinite dimensional and separable Hilbert space.   

Recall that $T\in\LH$ is bounded below if there exists $C>0$ such that $\|Tx\|\geq C\|x\|$ for all $x\in \GH$. Equivalently, $T$ is bounded below if and only if $T^*T$ is invertible. 
In the sequel, the spectral radius of $T$ is denoted by $r(T)$. 

In order to state a theorem following from the work of Shimorin \cite{shimorin}, we introduce the following definitions. 

\begin{defn}
	Let $T\in\LH$. 
	\begin{enumerate}
		\item The operator $T$ is \emph{pure} if $\bigcap_{n\geq 0}T^n \GH=\{0\}$.
		\item The operator $T$ has  the \emph{wandering subspace property} if $\GH$ is the closed linear hull (span) of $\{T^nE:n\geq 0\}$, where $E:=\GH\ominus T\GH$.
		\item For $T\in\LH$ bounded below, its \emph{Cauchy dual} is denoted $T'$ and defined by $T':= T(T^*T)^{-1}$. 
	\end{enumerate} 
\end{defn}

\begin{defn}
	Let $D=D(0,r)$ be the open disc of $\C$ centered at $0$ and of radius $r>0$. Let $E$ be a Hilbert space and  let $\Hil$ be a Hilbert space of holomorphic functions on $D$ taking values in $E$. A reproducing kernel on $\Hil$ is a map 
	\[ k~:~ \left\lbrace \begin{array}{ccc}
	D \times D & \rightarrow & \B (E) \\ 
	(\lambda, z) & \mapsto & k(\lambda, z)
	\end{array}   \right. \]
	such that
	\begin{enumerate}
		\item $\forall \lambda \in D$, $\forall e \in E$, $k(\lambda, \cdot) e \in \Hil$ ;
		\item $\forall \lambda \in D$, $\forall f \in \Hil$, $\forall e \in E$, $\langle f, k(\lambda, \cdot)  e \rangle_{\Hil} = \langle f(\lambda), e\rangle_E$.
	\end{enumerate}
\end{defn}

\subsection{Unitary equivalence with a shift on a reproducing kernel Hilbert space}
The following theorem is a consequence of the work of Shimorin  \cite{shimorin} but not stated explicitly. For completeness we will prove it in detail, by putting together the ideas developed in \cite{shimorin}. 

\begin{thm}\label{th:shimorin}
	Let $T\in\LH$ such that $T$ is bounded below, pure and with the wandering subspace property. 
	Then, there exists  a reproducing kernel Hilbert space $\Hil$ of holomorphic functions from  $D(0,r)$, where $r=r(T')$, to $E=\GH\ominus T\GH$, and a unitary operator $U : \GH \rightarrow \Hil$ such that
	
	\[  T = U^{-1} \Sigma U ,\]
	
	where $\Sigma ~:~ \left\lbrace \begin{array}{lll}
	\Hil & \rightarrow  & \Hil \\ 
	f & \mapsto  & (z \mapsto z ~ f(z)) 
	\end{array}  \right.  \in \B(\Hil) $.
	Moreover, the reproducing kernel $k$ is such that $k(0, \cdot) = (z \mapsto \Id_{{\mathcal L}(E)})$.
\end{thm}
\begin{proof}
	\textbf{We first construct $U$}. \\Since $T$ is bounded below, its Cauchy dual $T'$ is well defined. Denote by $L$ the adjoint of $T'$ and denote by $P$ the orthogonal projection onto $E$. \\
	\textbf{Claim 1}: $P=\Id-TL$.\\
	Indeed, let $Q = \Id - TL$. Since $LT = \Id$, it follows that 
	\[Q^2 = I - 2 TL + TLTL = \Id - TL = Q.\]
	Moreover, $Q$ is a self-adjoint operator since $TL$ is self-adjoint. It suffices to show that $\ker (TL) = E$. Since, $T$ is left invertible, we get
	\[\ker (TL) = \ker (L) = \ker ((T^*T)^{-1}T^*) = \ker(T^*)= (TH)^\perp = E. \]
	We now define the linear mapping  $U$ in the following way:
	\[ U ~:~ \left\lbrace  \begin{array}{lll}
	\Hil & \rightarrow & \text{Hol }(D(0,r);E)  \\ 
	x & \mapsto  & \sum\limits_{n \geqslant 0} \left( PL^nx \right)z^n 
	\end{array}  \right.  .\]
	The convergence  of the series follows from the fact that $r$ is the spectral radius of $L$.  \\
	\textbf{Claim 2}: $U$ is one-to-one.\\
	Indeed, let $x \in \ker (U)$. Then, for every $n \in \N$, $PL^nx = 0$. We prove that $x \in\bigcap\limits_{n \geqslant 1} T^n \GH$. Let $n \geqslant 1$ and note  that, according to Claim 1, 
	\[
	\sum\limits_{k=0}^{n-1} T^kPL^kx  =  \sum\limits_{k=0}^{n-1} T^kL^k x - T^{k+1}L^{k+1} x \\
	=  x - T^nL^n x.\]
	It follows that 
	\[
	x =  x - T^nL^n x + T^nL^n x \\
	=  \sum\limits _{k=0}^{n-1} T^kPL^kx + T^nL^n x.\]
	Since  for all $k \in \{ 0 ; n-1 \}$, $PL^kx = 0$, we get  $x = T^n L^n x \in T^nH$.  \\
	
	Let $\Hil \subset \text{Hol }(D(0,r);E)$ be the image of $U$. Since $U$ is one-to-one, $U$ is an isomorphism of vector spaces. We define on $\Hil$ a scalar product by setting
	\[\forall f,g \in \Hil, ~ \langle f,g\rangle_\Hil = \langle U^{-1}f,U^{-1}g\rangle_H,\]
	so that $U$ is unitary.\\
	The second step consists in checking that \textbf{$\Hil$ is a reproducing kernel Hilbert space.}\\
	For $\lambda \in D(0,r)$ and $e \in E$, we have
	\begin{eqnarray*}
		\langle f(\lambda),e \rangle_E & = & \langle  \sum\limits_{n \geqslant 0} \left( PL^nU^{-1}f \right) \lambda^n , e \rangle_E  =   \langle  \sum\limits_{n \geqslant 0} (\lambda L)^n (U^{-1}f), Pe \rangle_E\\
		& = &  \langle  (\Id - \lambda L^{-1}) (U^{-1}f),e \rangle_E 
		=  \langle U^{-1}f , (\Id - \overline{\lambda}L^*)^{-1}e \rangle_H \\
		& = & \langle f, U(\Id - \overline{\lambda}L^*)^{-1}e\rangle_\Hil
	\end{eqnarray*}
	
	On the other hand, for $z\in D(0,r)$, we have:
	\begin{eqnarray*}
		\left( U(\Id - \overline{\lambda}L^*))^{-1}e \right) (z) & = & \sum\limits_{n \geqslant 0}PL^n \left[ (\Id - \overline{\lambda}L^*)^{-1} e \right]z^n\\
%		&= & P \left( \sum\limits_{n \geqslant 0}L^n \left[ (\Id - \overline{\lambda}L^*)^{-1} e \right]z^n \right)\\
		& = & P \left( \sum\limits_{n \geqslant 0}(zL)^n \left[ (\Id - \overline{\lambda}L^*)^{-1} e \right] \right)\\
		& = &  P (\Id-zL)^{-1}(\Id - \overline{\lambda}L^*)^{-1} e.
	\end{eqnarray*}
	Therefore $\Hil$ is a reproducing kernel Hilbert space of holomorphic functions, whose reproducing kernel is defined by
	\[ k(\lambda, z) = P (\Id-zL)^{-1}(\Id - \overline{\lambda}L^*)^{-1} . \]
	\textbf{The third step} consists in proving that  \textbf{$z \mapsto k(0,z)$ is a constant function} whose value is $\Id_E$. To that aim we prove that, for every $f \in \Hil$ and every $e \in E$,
	\[ \langle f,k(0,\cdot)e \rangle_\Hil  = \langle f(0), e \rangle_E. \]
	Let $f \in \Hil$ and $x = U^{-1}f$. Let $e \in E$. Note that
	\begin{eqnarray*}
		f(0)&=& Px 
		 =  \langle Px, e \rangle_E 
		 =  \langle x , e \rangle_H 
		=  \langle f, Ue \rangle_\Hil .
	\end{eqnarray*}
	However, by Claim 1, $P e = e = e - TL e$. Hence, $TLe = 0$ and so $Le = 0$ since $T$ is bounded below. Therefore $U e = \sum\limits_{n \geqslant 0} PL^nx z^n = e$, that is, $Ue$ is the constant function $z \mapsto e$. Then,
	\[ \langle f, k(0,\cdot)e \rangle_\Hil = \langle f, z \mapsto e \rangle. \]
	
	\textbf{The last step} consists in proving that \text{$T$ is unitarily equivalent to $\Sigma$}. 
	%We prove that $\Sigma U = UT$.
	Let $x \in H$. Let $z \in D(0,r)$.
	\begin{eqnarray*}
		(UT)(x)(z) & = & \sum\limits_{n \geqslant 0} P (L^nTx)z^n 
		=  \sum\limits_{n \geqslant 1} P (L^nTx)z^n \\
		& = & z \sum\limits_{n \geqslant 1} P L^{n-1}(LT)x z^{n-1}
		=  z \sum\limits_{n \geqslant 1} P L^{n-1}x z^{n-1}\\
		& = & z U(x)(z)
		=  \Sigma U (x) (z).
	\end{eqnarray*} 
	This concludes the proof of the theorem.
\end{proof}
We can now obtain a representation theorem for $C_0$-semigroups whose cogenerator satisfies the hypothesis of the previous theorem.  
\begin{cor}\label{cor:semigroups}
	Let $(T_t)_{t\geq 0}$ be a $C_0$-semigroup on $\GH$ which admits a cogenerator $V$. Assume that $V$ is bounded below, pure and with the wandering subspace property.
	Then, there exists  a reproducing kernel Hilbert space $\Hil$ of holomorphic functions from 
	$D(0;r) \rightarrow E$ (with $r=r(V')$ and $E=\GH\ominus V\GH$) and a unitary operator $U : \GH \rightarrow \Hil$ such that, for every $t \geqslant 0$,
	\[T_t = U^{-1}S_tU\]
	where $S_t : \left\lbrace \begin{array}{lll}
	\Hil & \rightarrow & \Hil \\ 
	f & \mapsto & \left(z \mapsto e^{t\frac{1+z}{1-z}} f(z)\right)
	\end{array} \right. \in \B(\Hil)$.
\end{cor}
\begin{proof}
	By Theorem~\ref{th:shimorin} applied to $V$, there exist $\Hil$ and $U$ such that $V = U^{-1} \Sigma U$. Let $(S_t)_{t\geq 0}$ be the $C_0$-semigroup whose cogenerator is $\Sigma$ and the generator is $A$. 
	We have that $A = (\Sigma + \Id)(\Sigma-\Id)^{-1}$. Since $\Sigma$ is a multiplication operator on $\Hil$, $A$ is also a multiplication operator on $\Hil$ and, for every $f \in \dom (A)$, and for every $z \in D(0,r)$,
	\[ A(f)(z) = \frac{z+1}{z-1} f(z) . \]
	We now prove that $S_t$ is the multiplication operator whose symbol is $z \mapsto e^{t\frac{z+1}{z-1}}$. Since $\dom(A)$ is dense in $\Hil$, it suffices to show that, for every $f \in \dom(A)$,
	\[ \forall z \in D, ~ S_t(f)(z) = e^{t\frac{z+1}{z-1}} f(z) . \]
	Let $f \in \dom(A)$ and 
	\[ \Phi ~:~ \left\lbrace \begin{array}{lll}
	\R^+& \rightarrow & \Hil \\ 
	t & \mapsto  & S_t(f)
	\end{array}   \right. . \]
	Note that $\Phi$ is differentiable, and for every $t\geq 0$, $\Phi^\prime (t) = A (\Phi(t))$.
	Let $z \in D(0,r)$. We prove that for all $e\in E$, 
	\[ \langle S_t(f)(z),e\rangle_E = \left\langle e^{t\frac{z+1}{z-1}} f(z), e \right\rangle_E . \]
	Let $e \in E$ and 
	\[ \varphi ~:~ \left\lbrace \begin{array}{lll}
	\R^+& \rightarrow & \C \\ 
	t & \mapsto  & \langle S_t(f)(z), e \rangle_E
	\end{array}   \right. . \]
	For every $t\geq 0$,
	\begin{eqnarray*}
		\phi(t) & = &\langle S_t(f),k(z,\cdot)e \rangle_\Hil\\
		& = & \langle \Phi(t),k(z,\cdot)e \rangle_\Hil.
	\end{eqnarray*}	
	Since $\Phi$ is differentiable, $\phi$ is also differentiable and, for every $t\geq 0$,
	\begin{eqnarray*}
		\phi^\prime(t) & = & \langle \Phi^\prime (t), k(z,\cdot)e \rangle_\Hil 
		=  \langle A \phi(t), k(z, \cdot)e \rangle_E \\
		& = & \langle A \Phi(t)(z), e\rangle_E 
		= \left\langle \frac{z+1}{z-1} \Phi(t)(z), e \right\rangle_E \\
		& = & \frac{z+1}{z-1} \phi(t).
	\end{eqnarray*}
	Furthermore, $\phi(0) = \langle S_0(f)(z), e \rangle_E = \langle f(z),e\rangle_E$. Hence, $\phi$ is the solution of a linear Cauchy problem of order 1, which gives that, for every $t\geq 0$,
	\begin{eqnarray*}
		\phi (t) & = & e^{t\frac{1+z}{1-z}} \phi(0), \mbox{ and then }\\
		\langle S_tf(z),e \rangle_E & =  & e^{t\frac{z+1}{z-1}}  \langle f(z), e \rangle_E = 
 \langle e^{t\frac{z+1}{z-1}}   f(z), e \rangle_E.
	\end{eqnarray*}
	This concludes the proof.
\end{proof}	

\subsection{Semigroups modelled by a shift}
The aim of this section is to produce explicit examples on which Corollary~\ref{cor:semigroups} can be used. We first recall some definitions.  
\begin{defn}
Let $\GH$ be a 	complex infinite dimensional separable Hilbert space.  
\begin{enumerate}
	\item $T\in\LH$ is called a \emph{2-isometry} if $T^{*2}T^2-2T^*T+\Id=0$ (i.e. $\forall x\in\GH$, $\|T^2x\|^2+\|x\|^2=2\|Tx\|^2$). 
	\item $T\in\LH$ is called a \emph{2-contraction} if $T^{*2}T^2-2T^*T+\Id\geq 0$ (i.e. $\forall x\in\GH$, $\|T^2x\|^2+\|x\|^2\geq 2\|Tx\|^2$). 
	\item $T\in\LH$ is   \emph{concave} if $T^{*2}T^2-2T^*T+\Id\leq 0$ (i.e. $\forall x\in\GH$, $\|T^2x\|^2+\|x\|^2\leq 2\|Tx\|^2$). 
\end{enumerate}
\end{defn} 
Note that  the set of $2$-isometries is the intersection of the sets of concave and $2$-hypercontractive operators. 
Moreover an isometry is a $2$-isometry but the converse is false since the shift on the Dirichlet space $\mathcal D$ is a $2$-isometry but it  is not isometric (cf. \cite{gallardo-partington}).
\begin{thm}\label{th:sg-concave}
	Let $(T_t)_{t\geq 0}$ be a $C_0$-semigroup on $\GH$ such that for every $t>0$, $T_t$ is pure and concave. Then  there exist $r>0$, a Hilbert space $E$  and a reproducing kernel Hilbert space $\Hil$ of holomorphic functions from 
	$D(0;r)$ into $E$  and a unitary operator $U : \GH \rightarrow \Hil$ such that, for every $t \geqslant 0$,
	\[T_t = U^{-1}S_tU\]
	where $S_t : \left\lbrace \begin{array}{lll}
	\Hil & \rightarrow & \Hil \\ 
	f & \mapsto & \left(z \mapsto e^{t\frac{1+z}{1-z}} f(z)\right)
	\end{array} \right. \in \B(\Hil)$. 
\end{thm}
The proof of  Theorem~\ref{th:sg-concave} relies on several lemmas stated below.

\begin{lem}\label{existscogenerator}
	Let $(T_t)_{t\geq 0}$ be a $C_0$-semigroup such that $T_1$ is concave. Then, $(T_t)_{t\geq 0}$ has  a cogenerator.
\end{lem}
\begin{proof}
	This is a very slight adaptation of the proof of Lemma 2.1 in \cite{gallardo-partington}. Let $A$ be the generator of $(T_t)_{t\geq 0}$. The growth bound $\omega$ of $(T_t)_{t\geq 0}$, defined by
	\[ \omega= \inf\left\lbrace w \in \R :\exists M \geqslant 1 \text{ such that } \forall t\geq 0, \|T_t\| \leqslant M e^{w t} \right\rbrace   \]
	is such that, for every $t>	0$, 
		\[\omega = \frac{1}{t}\log (r(T_t)) \]
	where $r(T_t)$ is the spectral radius of $T_t$. Moreover, we have
	\[ \sup\left\lbrace \Re(\lambda) ~ \vert ~ \lambda \in \sigma(A) \right\rbrace \leqslant \omega \]
		(see for instance \cite{EN}, Chap. IV, Section 2, Prop. 2.2). To prove that the cogenerator is well-defined, it suffices to show that $ \omega < 1$ (since then $1 \in \rho (A)$). We show that $r(T_1) \leqslant 1$. This comes from the fact that, since $T_1$ is concave, then, for every $n \in \N^*$,
	\[ \|T_1| \leqslant \sqrt{1 + \left( \|T_1\| + 1 \right) n} \]
	Then, $r(T_1) = \lim\limits_{n \rightarrow + \infty} \|T_1^n\|^\frac{1}{n} \leqslant 1$, and thus $\omega \leqslant 0$, which concludes the proof.
\end{proof}

\begin{lem}\label{equivalence}
	Let $(T_t)_{t\geq 0}$ be a $C_0$-semigroup which has a cogenerator $V$. Let $A$ be its generator. The following assertions are equivalent
	\begin{enumerate}[label=(\roman{*})]
		\item for every $t\geq 0$, $T_t$ is concave;
		\item $\forall x \in \GH$, $\phi_x : t \mapsto \|T_t x\|^2$ is concave;
		\item $\forall y \in D(A^2)$, $\Re \left( \langle A^2 y, y \rangle \right) + \|Ay\|^2 \leqslant 0$;
		\item $V$ is concave.
	\end{enumerate}
\end{lem}
\begin{proof}
The  proof uses similar methods to those   of \cite[Prop. 2.6]{jacob}.  For the sake of completeness we give the details. \\
$(i)\Rightarrow(ii)$: Let $x \in \GH$ and  $\phi_x : t \mapsto \|T_t x\|^2$. We prove that
\[  \forall  t\geq 0, ~ \forall \tau \geqslant 0,~ \phi_x(2\tau + t) + \phi_x(t) \leqslant 2 ~ \phi_x (t+\tau). \]
Let $t\geq 0$ and $\tau \geqslant 0$. Since $T_\tau$ is concave, one has
\[\|T_\tau^2 T_t x\|^2 + \|T_t x\| \leqslant 2 \|T_\tau T_t x\|,  \]
which is the above inequality. Since $\phi_x$ is continuous, it follows that $\phi_x$ is concave.\\
$(ii) \Rightarrow (i)$: Let $t\geq 0$ and $x \in \GH$. Since $t \mapsto \|T_t x\|^2$ is concave, we get
\[ \|T_{\left( \frac{1}{2}0 + \frac{1}{2}2t \right)}x\|^2 \geqslant \frac{1}{2} \left( \|T_0 x\|^2 + \|T_{2t}x\|^2 \right),  \]
that is,
\[ 2 \|T_t x\|^2 \geqslant \|x\|^2 + \|T_t^2 x\|^2 . \]	
\noindent $(ii) \Rightarrow (iii)$: Let $y \in D(A^2)$. Then, the function $\phi_y : t \mapsto \|T_t y\|^2$ is twice differentiable and, for every $t\geq 0$,
\begin{eqnarray*}
	\phi_y^{\prime \prime}(t) & = & \langle A^2 T_t y, T_t y \rangle + 2 \langle A T_t y, A T_t, y \rangle + \langle T_t y, A^2 T_t y \rangle \\
	& = & 2 \left( \Re \left(\langle A^2 T_t y, T_t y \rangle \right) + \|AT_ty\|^2  \right).
\end{eqnarray*}
Taking $t=0$, one gets
\[ \phi_y^{\prime \prime}(0) = 2 \left( \Re \left( \| A^2  y,  y \| \right) + \|Ay\|^2  \right). \]
Since $\phi_y$ is concave, $\phi^{\prime \prime}(0) \leqslant 0$, which gives the result.\\
$(iii) \Rightarrow (ii)$: We prove first that, for every $y \in D(A^2)$, $\phi_y$ is concave. Let $y \in D(A^2)$. Note that $\phi_y$ is twice differentiable. Let $t\geq 0$. Note that $T_t y \in D(A^2)$, and that $\phi_y^{\prime \prime}(t)=\langle A^2 T_t y, T_t y \rangle + 2 \langle A T_t y, A T_t, y \rangle + \langle T_t y, A^2 T_t y \rangle  \leqslant 0$. Hence, $\phi_y$ is concave.

We now prove the result for every $x \in H$. For $x \in H$, we show that $\phi_x : t \mapsto \|T_t x\|^2$ is concave. Let $t,s \in \R^+$. Let $\tau \in [0;1]$. Since $D(A^2)$ is dense in $H$ (in fact, $\bigcap\limits_{n \geqslant 1}D(A^n)$ is dense in $\GH$ (see \cite[Chap. 3, Thm 3.2.1]{staffans}, there exists a sequence $(y_n)_n$ such that, for every $n \in \N$, $y_n \in D(A^2)$ and $y_n \rightarrow x$. However, for every $n \in \N$, $\phi_{y_n}$ is concave so
\[ \phi_{y_n}((1 - \tau)t + \tau s) \geqslant (1 - \tau) \phi_{y_n}(t) + \tau \phi_{y_n}(s)  \]
and hence
\[ \|T_{(1-\tau)t+ \tau s}y_n\|^2 \geqslant (1 - \tau) \|T_t y_n\|^2 + \tau \|T_s y_n\|^2.  \]
Letting $n \rightarrow + \infty$, we get
\[ \|T_{(1-\tau)t+ \tau s}x\|^2 \geqslant (1 - \tau) \|T_t x\|^2 + \tau \|T_s x\|^2, \]
which concludes the proof.\\
$(iii) \Rightarrow (iv)$: Let $x \in H$. Let $y = (A-I)^2x$. Note that $y \in D(A^2)$. Then,
\begin{eqnarray*}
	\|V^2x\|^2 + \|x\|^2 - 2 \|Vx\|^2 & = & \|(A+I)^2y\|^2 + \|(A-I)^2y\|^2 - 2 \|(A^2 - I)y\|^2 \\
	& = & 4 \langle A^2y, y \rangle + 8 \langle Ay, Ay \rangle + 4 \langle y, A^2y \rangle \\
	& = & 8 \left( \Re \langle A^2y, y \rangle + \|Ay\|^2 \right)\\
	& \leqslant & 0.
\end{eqnarray*}
Hence, $V$ is concave.\\
$(iv) \Rightarrow (iii)$: The previous calculation shows that
\[ \forall y \in \im (A^2), ~ \Re \langle A^2y, y \rangle + \|Ay\|^2 \leqslant 0.  \]
However $(A-I)^2$ is a bounded linear operator such that there exists an \emph{a priori} unbounded operator $T$ with dense domain such that $(A-I)^2T = I$. Hence, $ \im (A-I)^2$ is dense in $H$. We then get
\[ \forall y \in H, ~ \Re \langle A^2y, y \rangle + \|Ay\|^2 \leqslant 0,  \]
which concludes the proof.
\end{proof}

The next result is Theorem~3.6 in \cite{shimorin}.
\begin{lem}\label{wold}
	Every concave operator has a Wold-type decomposition. In particular, every pure concave operator has the wandering subspace property.
\end{lem}
The last step is the following. 
\begin{lem}\label{pure}
	Let $(T_t)_{t\geq 0}$ be a $C_0$-semigroup which has a cogenerator $V$. Assume that $V$ has a Wold-type decomposition. Assume that, for every $t > 0$, $T_t$ is pure. Then, $V$ is pure.
\end{lem}
\begin{proof}
The proof mimics the proof of \cite[Prop. 2.5]{gallardo-partington}. 
The spaces $\GH_1 = \bigcap_{n \geqslant 0} V^n \GH$ and $\GH_2$ defined as the closed linear hull of $\left\lbrace V^n (\GH \ominus V\GH ) \right\rbrace$ are two closed subspaces of $\GH$ invariant by $V$ such that $\GH = \GH_1 \oplus \GH_2$, $U:= \restriction{V}{\GH_1} \in {\mathcal L}(\GH_1)$ is unitary and $S := \restriction{V}{\GH_2} \in {\mathcal L}(\GH_2)$ has the wandering subspace property (this is the Wold decomposition of $V$, see Lemma \ref{wold}). 
We want to prove that $\GH_1 = \left\lbrace 0 \right\rbrace$.\\
	Note that, for every $\tpos$, $T_t$ and $A$ commute. From this, we deduce that, for every $\tpos$, $T_t$ and $V = (A+I)(A-I)^{-1}$ commute. Let $\tpos$. We show that $\GH_1$ is invariant by $T_t$. Let $x \in \GH_1 = \bigcap\limits_{n \geqslant 0} V^n \GH$. Let $n \geqslant 0$. Since $x \in V^n \GH$, there exists $y \in \GH$ such that $x = V^n y$. Then, $T_t x = T_t V^n x = V^n T_t x \in V^n \GH$, which proves that $T_t x \in \GH_1$.\\
	We now consider the semigroup $(\tilde{T_t})_\tpos$ induced by $T_t$ on $\GH_1$. Let $B$ be the generator of $(\tilde{T_t})_\tpos$. Note that $B$ is the restriction of the generator $A$ of $(T_t)_\tpos$ to $\dom (A) \cap \GH_1$ (which is a dense subspace of $\GH_1$, see \cite{EN}, Chapter 2, Section 2). The cogenerator of $(\tilde{T_t})_\tpos$ is $U$. Since $U$ is unitary, $B$ is skew-adjoint (that is, $B^* = - B$), and hence, for every $\tpos$, $\tilde{T_t}$ is unitary. However, for every $\tpos$, $T_t$ is pure. This proves that $\GH_1 = \left\lbrace 0 \right\rbrace$.

\end{proof}

Combining these lemmas with Corollary \ref{cor:semigroups}, we have completed the proof of Theorem
\ref{th:sg-concave}. \\ \fini

\bibliography{biblio-universal}

\end{document}